\documentclass[10pt, letterpaper]{article}

\usepackage{graphicx}
\usepackage[linesnumbered,boxruled,vlined]{algorithm2e} 
\usepackage{amsmath}
\usepackage{amssymb}
\usepackage{comment}
\usepackage{scrextend}
\usepackage{mathrsfs}
\usepackage{amsthm}
\newtheorem{proposition}{Proposition}[section]
\newtheorem{theorem}{Theorem}[section]

\newtheorem{example}{Example}
\newtheorem{remark}{Remark}[section]
\DeclareMathOperator*{\argmax}{arg\,max}
\DeclareMathOperator*{\argmin}{arg\,min}
\DeclareMathAlphabet{\altmathcal}{OMS}{cmsy}{m}{n}
\usepackage{tikz}
\usetikzlibrary{shapes.geometric, arrows}
\tikzstyle{rootnode} = [rectangle, rounded corners,   
minimum width=0.3cm, 
minimum height=0.3cm,
text centered, 
draw=black, 
fill=orange!30]
\tikzstyle{node1} = [rectangle, rounded corners, 
minimum width=0.4cm, 
minimum height=0.4cm,
text centered, 
draw=black, 
fill=gray!30]
\tikzstyle{node2} = [rectangle, rounded corners,  
minimum width=0.4cm, 
minimum height=0.4cm,
text centered, 
draw=black, 
fill=green!30]
\tikzstyle{node3} = [circle,  
minimum width=0.3cm, 
minimum height=0.3cm,
text centered, 
draw=black, 
fill=red!40]
\tikzstyle{node4} = [circle,  
minimum width=0.3cm, 
minimum height=0.3cm,
text centered, 
draw=black, 
fill=yellow!30]
\tikzstyle{arrow} = [thick,>=stealth]

\providecommand{\keywords}[1]
{
  \small	
  \textbf{\textit{Keywords---}} #1
}

\usepackage[caption=false]{subfig}
\usepackage{hyperref}
\usepackage{comment}
\usepackage{float}
\title{An Algorithm to Solve Cardinality Constrained Quadratic Optimization Problem with an Application to the Best Subset Selection in Regression}
\author{Vikram Singh$^{1}$, Min Sun$^{2}$  \\
        \small $^{1}$University of Central Oklahoma, Edmond, OK \\
                 \small  vsingh2@uco.edu \\
        \small $^{2}$The University of Alabama, Tuscaloosa, AL \\
                 \small msun@ua.edu\\
}
\date{March 2025} 

\begin{document}

\maketitle

\begin{abstract}
A lot of problems, from fields like sparse signal processing, statistics, portfolio selection, and machine learning, can be formulated as a cardinality constraint optimization problem. The cardinality constraint gives the problem a discrete nature, making it computationally challenging to solve as the dimension of the problem increases. In this work, we present an algorithm to solve the cardinality constraint quadratic optimization problem using the framework of the interval branch-and-bound. Interval branch-and-bound is a popular approach for finding a globally optimal solution in the field of global optimization. The proposed method is capable of solving problems of a wide range of dimensions.  
In particular, we solve the classical best subset selection problem in regression and compare our algorithm against another branch-and-bound method and GUROBI's quadratic mixed integer solver. Numerical results show that the proposed algorithm outperforms the first and is competitive with the second solver.
\end{abstract}
\section{Introduction}
In recent years, there has been a growing interest among researchers to solve optimization problems subject to a Cardinality Constraint (CC), perhaps because of its applications in various fields like high-dimensional statistics, machine learning, and sparse portfolio selection (see \cite{tillman2024cardinality}). 
In the field of sparse signal approximation, the goal is to find a sparse vector $x \in \mathbb{R}^{p}$ which fits the model $b=Ax+e$, where $A\in \mathbb{R}^{n \times p}$ is the matrix of linear measurements with $n << p$ and $e$ represents noise (see \cite{blumensath2008iterative}). This problem can be formulated as minimizing the quadratic function $\|b-Ax \|_{2}^{2}$ subject to a CC. Another example is the classical Best Subset Selection (BSS) in regression, which requires selecting a small number of predictors to be included in the model that minimizes the residual sum of squares. Motivated from these applications, we 
consider the following Cardinality Constrained Quadratic Optimization (CCQO) problem
\begin{equation}\label{ccqo}
    \min_{x \in B } \;\; \frac{1}{2}x^{T}Q x + q^{T}x +c \quad \text{subject to} \quad \|x\|_{0} \leq k, \tag{CCQO}
\end{equation}
where $ Q \in \mathbb{R}^{p \times p}$ is a symmetric positive semi-definite matrix, $q \in \mathbb{R}^{p}$, $B=[\underline{B},\overline{B}]$ is a $p$-dimensional box with $\underline{B},\overline{B}\in \mathbb{R}^{p}$, $k<p$ is a positive integer, and $\|\cdot\|_{0}$ is a pseudo-norm which gives the number of nonzero entries in a vector.
Note that even though the objective function is convex, the CC, $\|x\|_{0} \leq k$ makes \eqref{ccqo} non-convex and NP-hard (see \cite{natarajan1995sparse}). Due to the discrete nature of CC, we cannot use the existing techniques from the vast literature of continuous optimization to solve \eqref{ccqo}, therefore designing new procedures is crucial. 

In this paper, we adapted the framework of the interval branch-and-bound (IBB) method with novel modifications to solve the \eqref{ccqo} problem. IBB is a well-known method to solve global optimization problems even when the feasible set is non-convex, and the goal is to find all the minimizers of the problem. The new algorithm converges to the optimal solution of \eqref{ccqo} in a finite number of iterations. We demonstrate the efficiency of our algorithm by solving the BSS problem up to dimension 2000 using synthetic data.
The rest of the paper is organized as follows. In section \ref{sec:intro-to-bss}, we briefly introduce the BSS problem, along with two existing methods to solve it. 
In section \ref{sec:bss-using-ivlbb}, we present a new algorithm based on IBB to solve \eqref{ccqo}. Section \ref{sec:ibb_for_bss} shows the application of the proposed algorithm to solve the BSS problem. We present the numerical results in section \ref{sec:test_results} followed by the conclusion in section \ref{sec:discussion}.
\section{Best subset selection in regression}\label{sec:intro-to-bss}
Consider a linear regression model \( y=X\beta+\varepsilon \),
where \(y \in \mathbb{R}^{n}\) is a response vector, \(X \in \mathbb{R}^{n \times p}\) is a design matrix, \(\beta \in \mathbb{R}^{p}\) is an unknown coefficient vector, and \(\varepsilon \in \mathbb{R}^{n}\) is a noise vector. The columns of $X$ have been standardized to have zero mean and unit $l_{2}$-norm. 
One common objective is to find a desired coefficient vector $\beta$ by minimizing the residual sum of squares (RSS). This is the so-called ordinary least squares problem
\begin{equation}\label{ols}
  \min_{\beta \in \mathbb{R}^{p} } \; \; \|y-X\beta\|_{2}^{2}. \tag{OLS}  
\end{equation}
\eqref{ols} is an unconstrained convex optimization problem and can be solved efficiently by many existing optimization algorithms. In particular, if $X$ has full rank, its optimal solution is uniquely determined by $\Hat{\beta}=(X^{T}X)^{-1}X^{T}y$.  

However, to better interpret the underlying data, we want to choose a model with only $k$ (out of $p$) predictors that would fit the data well. This task gives rise to the following BSS problem
\begin{equation}\label{bss}
    \min_{\beta \in \mathbb{R}^p} \; \; \|y-X\beta\|_{2}^{2} \quad \text{subject to} \quad \|\beta\|_{0} \leq k.  
    \tag{BSS}
\end{equation}
\eqref{bss} is well known to be computationally challenging to solve as the number of possible subsets grows rapidly with an increase in the dimension of the problem. Next, we present two methods to solve the \eqref{bss} problem.
\subsection{Branch and bound algorithm for \eqref{bss} } 
 In pattern recognition literature, the branch and bound (BB) algorithm for feature selection using a monotone criterion function was first introduced by \cite{narendra1977branch}, and it has become popular for solving the feature selection problem. Several variants have appeared since then (see \cite{fukunaga1990intro}, \cite{gatu:2006}, \cite{somol2004fast}). 
In particular, \cite{fukunaga1990intro} introduced an in-level node ordering to improve the BB algorithm by eliminating undesired features at early stages. Another improvement was given by \cite{yu1993more}, which provided a ``minimum-solution-tree'', a subtree of the original BB tree, that is enough to explore to get the optimal solution, saving some computational effort. Because of the unique structured nature of BB, it can be naturally visualized or represented by a tree. In particular, BB for \eqref{bss} has been represented by a regression tree (see \figurename{ \ref{fig:BBtree}}). 
\subsection{Mixed integer optimization formulation for \eqref{bss}}
In a recent work, \cite{bertsimasEtal:2015} converted  \eqref{bss} into a mixed integer optimization (MIO) problem by rewriting the cardinality constraint as
\[ -Mz\leq \beta \leq Mz, \quad  
  \sum_{i=1}^{p} z_{i} \leq k, \quad 
   z \in \{0,1\}^{p}, \]  
where $z$ is a vector of binary variables and $M$ is a technical uniform variable bound. The resulting problem can be solved using any MIO solver. In particular, they used two different problem formulations for $p\leq n$ and $p>n$ and demonstrated that solving \eqref{bss} problem in higher dimensions is within reach now. A special formulation for the $p>n$ case results in a problem with dimension $n$, which is particularly useful when $p \gg n$.
\section{Interval branch and bound to solve \eqref{ccqo}}
\label{sec:bss-using-ivlbb}
To the best of our knowledge, IBB has not been applied to solve \eqref{ccqo}, perhaps because the CC is not included within the standard formulation of the constraint optimization problems solved by IBB. 
We would have to modify the IBB algorithm to make it possible to treat the CC effectively. 
We now highlight major features in our algorithm to take advantage of special properties of the \eqref{ccqo} problem.
\subsection{Branching} 
The IBB would normally require repeated partitions of the search domain $B$ to yield the desired global convergence. It is no longer necessary for the optimization problems with CC. In particular, we need to partition an interval only at $0$; if an interval does not include $0$, there is no need to partition the interval. If $0$ is strictly between $a$ and $b$, interval $[a, b]$ would be split into $[a, 0)$, $[0, 0]$, $(0, b]$.  If $0$ is equal to $a$ or $b$, $[a, b]$ would be split into two sub-intervals $[0,0]$ and $(0,b]$ or $[a,0)$ respectively. 
\subsection{Bounding} 
For a working box $V$, a standard way of bounding \(f(V)=\{f(x): x \in V \}\) is done by constructing an inclusion function \(F(\cdot)\) with \(\inf \>F(V)\) as an acceptable lower bound of \(f(V)\). Consequently, \(\inf \> F(V)\) could also be used to test bound-based deletion conditions. However, if partitions occur only at zeroes, the bound-based deletion condition won$'$t be effective since there is no chance to improve the accuracy of the bound after initial branching. In such a circumstance, we propose to use a tight lower bound of \(f(V)\), denoted by \( lb \, f(V) \), defined as
 \begin{equation*}
     lb\;f(V)=\min\;\{f(x): x \in V\}.
 \end{equation*}
\subsection{Deletion} 
In addition to the standard bound-based deletion condition, we apply the CC to remove additional sub-boxes. A working box $V$ is infeasible, or all its sub-boxes are infeasible (thus $V$ can be deleted) if it satisfies any one of these deletion conditions:
\begin{labeling}{AAAAA}  
    \item[($D1,V$)] \( \sum_{i=1}^{p} 1(0 \notin V_{i}) > \, k \), the number of intervals not containing $0$ is greater than $k$.
    \item[($D2,V$)] \( p-k \, < \sum_{i=1}^{p} 1(V_{i}=[0,0]) \), the number of degenerate intervals $[0, 0]$ is greater than $p-k$.
    \item[($D3,V$)] \( \sum_{i=1}^{p} 1(0 \notin V_{i}) = \, k \), the number of intervals not containing $0$ is equal to $k$.
    \item[($D4,V$)] \(  p - k = \sum_{i=1}^{p} 1(V_{i}=[0,0]) \), the number of degenerate intervals $[0, 0]$ is equal to $p-k$.    
\end{labeling}
With all these strategies included, we obtain Algorithm \ref{alg:ibb-bbss} (IBB$^{+}$) to solve \eqref{ccqo}.
\begin{algorithm}
\caption{IBB$^{+}$} 
\label{alg:ibb-bbss}
\KwIn{$p,k,B,f$}
\KwOut{$\beta^{*},f^{*}$}
    Take $Y:=B$.\\
    Find a feasible point $\Tilde{y}$ and set $\Tilde{f}=f(\Tilde{y})$, $\Tilde{\beta}=\Tilde{y}$.\\
    Set $y=lb \> f(Y)$. \\
    Initialize the list $L=\{(Y,\Tilde{y},y)\}$. \\
    \eIf{$L$ is empty \label{alg:ibb_selectbox}}{
       Set optimal point $\beta^{*}=\Tilde{\beta}$, optimal value $f^{*}=\Tilde{f}$, \Return. \\
    }{
    Select a node $(Y,\Tilde{y},y)$ from the list $L$. \label{alg:ibb_boxSelection} \\
    }
    Choose a coordinate direction $\eta$ such that $0 \in Y_{\eta}$ and $Y_{\eta} \neq [0,0]$. If there is no such coordinate direction, go to step \ref{alg:ibb_selectbox}.\\
    Partition the sub-box $Y_{\eta}$ at $0$ to get $r$ number of child boxes such that $Y=\bigcup\limits_{j=1}^{r} V_{j}$. \label{alg:ibb_partDir} \\
    Remove $(Y,\Tilde{y},y)$ from the list $L$.\\
    \For{$j=1$ to $r$}{If any of $(D1,V_{j})$ or $(D2,V_{j})$ hold, go to step \ref{alg:ibb_endbranchforloop}.\\
    If any of $(D3,V_{j})$ or $(D4,V_{j})$ hold, find a feasible point $\Tilde{v}_{j}$ such that $f(\Tilde{v}_{j})=lb \, f(V_{j})$, update $\Tilde{f}$ (also $\Tilde{\beta}$) if possible, go to step \ref{alg:ibb_endbranchforloop}.\\
    \textit{(Feasibility sampling)} Find a feasible point $\Tilde{v}_{j}$ in the box $V_{j}$. Update $\Tilde{f}$ (also $\Tilde{\beta}$) if possible. \\
    (\textit{Bound}) Calculate $v_{j}=lb \, f(V_{j})$. \label{alg:ibb_boundstep} \\
    If $\Tilde{f} < v_{j}$, go to step \ref{alg:ibb_endbranchforloop}. \\
    Add $(V_{j},\Tilde{v}_{j},v_{j})$ at the end of the list $L$. \\
    Go to the next iteration of $j$ loop.  \label{alg:ibb_endbranchforloop} \\
    }
    Go to step \ref{alg:ibb_selectbox}. \\
\end{algorithm}
Convergence of IBB$^{+}$ will not follow directly from the convergence analysis of the IBB due to the special cut at $0$ as the mesh of the partition will not get smaller. However, if tight $\inf \> F(V)$ is used, we still have global convergence to the optimal solution in a finite number of steps and monotone convergence to the optimal objective function value.
\begin{theorem}
\textup{IBB$^{+}$} reaches the optimal solution of \eqref{ccqo} in a finite number of iterations.
\end{theorem}
\begin{proof}
Each box in IBB$^{+}$ is visited exactly once. We either delete a selected box (hence discarding all of its child boxes) or add it to the list for further branching. As there are a finite number of boxes to visit, IBB$^{+}$ will find the optimal solution in a finite number of iterations. In the worst case, the standard bound-based deletion condition is never satisfied before reaching the optimal solution. In that event, the tree generated by the IBB$^{+}$ would contain all the feasible subsets and an optimal solution would be identified since $lb \> f(\cdot)$ is used in IBB$^{+}$. 
\end{proof}

In step \ref{alg:ibb_partDir} of IBB$^{+}$, the value of $r$ is 3 if $0$ is an interior point of $Y_{\eta}$ and $r$ is 2 if $0$ is at the boundary of $Y_{\eta}$. Also, we can choose more than one coordinate direction to partition our box.  
Due to the special branching in IBB$^{+}$, we do not have to use the standard interval box. Instead, we can use an integer flag to represent such a box in each coordinate direction. For $a<0$, $b>0$, the flag $0$ represents the degenerate interval $[0,0]$, the flag 1 represents an interval containing $0$ in it ($[a,0]$,$[0,b]$, $[a,b]$), and the flag $2$ represents an interval not containing $0$ in it ($[a,0)$ or $(0,b]$ or $[a,0) \cup (0,b]$). Therefore, a box $Y$ in our discussion can be interpreted as a multidimensional interval box or an integer flag vector. Under the flag interpretation, $r$ is always 2, and the initial box $B$ can be represented by a $p-$dimensional vector with all the components as integer $1$. For step \ref{alg:ibb_partDir}, we will partition the box $Y$ into 2 sub-boxes $V_{1}$ and $V_{2}$ such that $V_{1\eta}=0$ and $V_{2\eta}=2$. This simplified representation will save us computational time as well as memory space. A vector of such integer flags, along with the original search domain, would allow us to have a non-interval algorithm while preserving all the major convergence properties of the IBB method. 

Selecting a new node from the list (step \ref{alg:ibb_boxSelection}) to partition further is also a crucial step in any branch-and-bound algorithm. There are several selection criteria in the literature to select a node from the list (see \cite{morrison2016branch}); among them, Depth-First Search (DFS) and Best-First Search (BFS) are two popular choices. In our framework, BFS corresponds to selecting a node with the smallest $lb \, f(\cdot)$ value, and DFS corresponds to selecting a node with the maximum number of $2$ flags in the box. 

An additional property of IBB$^{+}$ can be stated in terms of $T(p, k)$ and $T_{n}(p, k)$, where $T(p, k)$ defines the underlying tree of IBB$^{+}$ for given $p$ and $k$ values, and $T_{n}(p, k)$ defines the number of nodes (including the root and leaf nodes) of $T(p, k)$.
\begin{proposition}
For a given $p$ and $k$, $T_{n}(p,k)=T_{n}(p-1,k)+T_{n}(p-1,k-1)+1 $.
\end{proposition}
\begin{proof}
The root node of $T(p,k)$ produces 2 child nodes. The one corresponding to flag 2 can be thought of as a root node for the tree $T(p-1,k-1)$ and the other one corresponding to flag 0 can be thought of as a root node for the tree $T(p-1,k)$. So, $T(p-1,k)$ and $T(p-1,k-1)$ are the subtrees of $T(p,k)$. Consequently, we get the above result.
\end{proof}
\section{Solving the Best Subset Selection using IBB$^{+}$}
\label{sec:ibb_for_bss}
The \eqref{bss} can be re-formulated as \eqref{ccqo} problem and can be written as 
\begin{equation}\label{bbss}
    \min_{\beta \in B} \;\; f(\beta)=\frac{1}{2}\beta^{T}Q \beta + q^{T} \beta +c \;\;\; \text{subject to} \;\;\; \|\beta\|_{0} \leq k,   \tag{BBSS} 
\end{equation}
where $Q=2X^{T}X$, $q=-2X^{T}y$, and $c=y^{T}y$. If the chosen box $B$ is big enough to contain the solution of \eqref{bss}, then \eqref{bss} and \eqref{bbss} will have the same optimal solution (see \cite{bertsimasEtal:2015}). Common reasons to use the additional box $B$ may include
\begin{itemize}
    \item ensuring a finite optimal solution;
    \item facilitating certain effective search procedures;
    \item incorporating any prior knowledge about the bounds of the unknown parameters.
\end{itemize}
In step \ref{alg:ibb_boundstep} of IBB$^{+}$, for a working box $V$, if \( \mathcal{I}=\{ i : V_{i} \neq 0, i=1,...,p \} \) is an index set, then  
\[ lb \, f(V)= \min_{\beta \in B_{\mathcal{I}}} \; \; \frac{1}{2}\beta^{T}Q_{\mathcal{I}}\beta + q_{\mathcal{I}}^{T}\beta+c ,   \]
where box $B_{\mathcal{I}}=\prod_{i\in \mathcal{I}} B_{i}$, $Q_{\mathcal{I}}$ is a submatrix of $Q$ with rows and columns indexed by $\mathcal{I}$, and $q_{\mathcal{I}}$ is a subvector of $q$ indexed by $\mathcal{I}$.
We can use any convex quadratic optimization solver to solve the above problem. In that case, no interval arithmetic would be used. Consequently, round-off errors have to be accepted. Thus, IBB$^{+}$ is considered as a branch-and-bound method but not officially an interval algorithm in the traditional sense.
\subsection{Feasibility sampling}
Finding good quality feasible points within IBB$^{+}$ helps with more bound based deletions and can accelerate the algorithm significantly to reach the optimal solution. We introduce Sequential Feature Swapping (SFS) given by \algorithmcfname{ \ref{alg:sfs}}, which is an iterative procedure that relies on the idea of swapping bad predictors from the currently selected model with good predictors not in the model at each iteration, starting from a model with $k$ predictors.  

 Let $\mathcal{A}$ be the set of indices of all the available predictors and $\mathcal{I}$ be the set of indices of the $k$ predictors already selected in the model. The set $\mathcal{A} \backslash \mathcal{I}$ represents the complement of the set $\mathcal{I}$ w.r.t. $\mathcal{A}$, and let
\[ q(\mathcal{I}) = \min_{\beta \in B_{\mathcal{I}}} \; \; \frac{1}{2}\beta^{T}Q_{\mathcal{I}}\beta + q_{\mathcal{I}}^{T}\beta+c. \]
 For an index $ s \in \mathcal{I}$, we define the gain in the function value as
\( G(s)=q(\mathcal{I}\backslash s) -q(\mathcal{I}). \)
For an index $ s \in \mathcal{A}\backslash \mathcal{I}$, we define the reduction in the function value as
\( R(s)=q(\mathcal{I})-q(\mathcal{I}\cup s). \)
At each iteration of SFS, we will drop the predictor with the minimum gain, and pick the predictor with the maximum reduction, to be included in the model.
\begin{algorithm}
\caption{SFS}
\label{alg:sfs}
 \KwIn{$k$, $\mathcal{A}$, $B$, $q$}
 \KwOut{$\mathcal{I}^{\ast}$}
      Choose an initial index set $\mathcal{I}$ of $k$ predictors.\\ 
     \textit{(Drop)}\label{alg:sfs_drop} Find the index $j \in \mathcal{I}$ such that
       \[  j=\argmin_{s \in \mathcal{I}} \; \; G(s),   
       \] and set $\hat{\mathcal{I}}=\mathcal{I} \backslash j $.      \\
       \textit{(Pick)} Find the index $i \in \mathcal{A}\backslash \mathcal{I}$ such that
       \[  i = \argmax_{s \in \mathcal{A} \backslash \mathcal{I}} \; \; R(s). \]   \\
       \eIf{$ q( \hat{\mathcal{I}}\cup i ) < q(\mathcal{I}) $}{
       \textit{(Switch)}
        Update $\mathcal{I}= \hat{\mathcal{I}} \cup i $. Go to step \ref{alg:sfs_drop}. \\
       }{Set $\mathcal{I}^{\ast}=\mathcal{I}$,\Return.
       }
\end{algorithm}
\begin{proposition}
Algorithm \ref{alg:sfs} terminates after a finite number of iterations.
\end{proposition}
\begin{proof}
At each iteration, for the updated set $\mathcal{I}$, the value of $q(\mathcal{I})$ decreases monotonically, and $q(\mathcal{I})$ is bounded from below by $q(\mathcal{I}^{*})$, where $\mathcal{I}^{*}$ is the optimal set of indices of $k$ predictors for the \eqref{bbss} problem. Hence, the Algorithm \ref{alg:sfs} terminates after a finite number of iterations.
\end{proof} 
We note that Algorithm \ref{alg:sfs} shares a similar idea as in Efroymson's stepwise algorithm (see \cite{miller2002subset}) with the difference that Algorithm \ref{alg:sfs} starts when we already have $k$ predictors selected in the model, whereas the Efroymson stepwise algorithm starts with no predictor in the model and sequentially selects a new predictor with a check included to drop some previously selected predictor from the model at each step. Additionally, see the ``splicing algorithm'' described in \cite{zhu2020polynomial}, which further generalizes the idea of Algorithm \ref{alg:sfs} by letting more than one predictor be swapped to pick a desirable sparse model. Swapping more than one predictor increases the chance of finding a better feasible point at the cost of an increase in the computation time.
\subsection{Finding $lb \, f(\cdot)$ using QR decomposition}
We can also use a recursive way of updating the $lb \, f(\cdot)$ of a child box using the $lb \, f(\cdot)$ of the parent box.
Suppose the design matrix $X$ has full column rank. Using QRD we have $Q^{T}X=R$, where $Q \in \mathbb{R}^{ n\times p}$ is an orthogonal matrix, $R \in \mathbb{R}^{p\times p}$ is an upper triangular matrix. For the initial box $Y$, 
$lb \, f(Y)=f(\Hat{\beta})$ where \( \Hat{\beta}= \argmin_{\beta} \> \|y-X\beta\|_{2}^{2}=R^{-1} Q^{T}y \). 
Suppose we partition the initial box $Y$ along the coordinate direction $\eta$ to get two child boxes. The child box with flag $2$ has the same $lb\, f(\cdot)$ as its parent box $Y$. The child box with flag $0$ corresponds to the linear model $y=X_{1}\beta +\varepsilon$ with the matrix $X_{1}$ obtained from $X$ by dropping its $\eta^{th}$ column, we can compute $lb \, f(\cdot)$ for this child box as follows. Update the $Q$ and $R$ matrices to get $Q_{1}^{T} X_{1}=R_{1}$ where $Q_{1}\in \mathbb{R}^{(p-1) \times (p-1)}$ is an orthogonal matrix and $R_{1} \in \mathbb{R}^{(p-1)\times (p-1)}$ is an upper triangular matrix (see \cite{watkins2010fundmatrixcomp} for updating $QRD$ after dropping a column from the $X$ matrix). Then, $lb\, f(V)=f(\beta_{1})$ where \( \beta_{1}= \argmin_{\beta} \> \|y-X_{1}\beta\|_{2}^{2}=R_{1}^{-1}Q_{1}^{T}y \).
We can modify $Q_{1}$ and $R_{1}$ to find $lb\,f(\cdot)$ for the child boxes of $V$ in a similar way. However, this procedure is not very efficient in practice, because we have to save the $Q$ and $R$ matrices for each box as attributes of a node, which uses a lot of memory space.
\subsection{IBB$^{+}$ versus BB}
\label{sec:ibb_vs_bb}
Both IBB$^{+}$ and BB are globally convergent algorithms. Figures \ref{fig:IvlBBtree} and \ref{fig:BBtree} show IBB$^{+}$ and BB trees for a small instance where we want to choose 2 predictors out of 5 available predictors, assuming no bound-based deletion condition has taken place and we are branching on the first available variable in a box. Each of the grey boxes in Figure \ref{fig:IvlBBtree} has the same $lb \, f(\cdot)$ value as their parent boxes. Thus in the IBB$^{+}$ tree, we do not have to compute $lb \, f(\cdot)$ for one of the child boxes unless it is a terminal box where the remaining flag 1s are changed to flag 0s. On the other hand, in the BB tree, we evaluate $lb \, f(\cdot)$ at each node by removing the feature given by the number inside the node one at a time, where $0$ represents the root node with all the features included. The red nodes in the BB tree show the nodes that can be skipped using the ``minimum-solution-tree'' strategy given by \cite{yu1993more}, saving $lb \, f(\cdot)$ evaluations. Including the root node in the BB tree, both trees use the same number of $lb \, f(\cdot)$ calls to reach the optimal solution without using any bound-based deletion. However, in practice, the number of $lb \, f(\cdot)$ calls for BB with ``minimum-solution-tree'' will be much higher than IBB$^{+}$ due to bound based deletions as shown by Table \ref{tab:ols_calls_ibb_vs_bb}. Further, BB works by deleting one feature at each node to reach a feasible solution represented by a leaf node at the bottom level, while IBB$^{+}$ uses the idea of selecting support for different features to enumerate all the feasible solutions, and there is no sense of level in IBB$^{+}$ tree. 
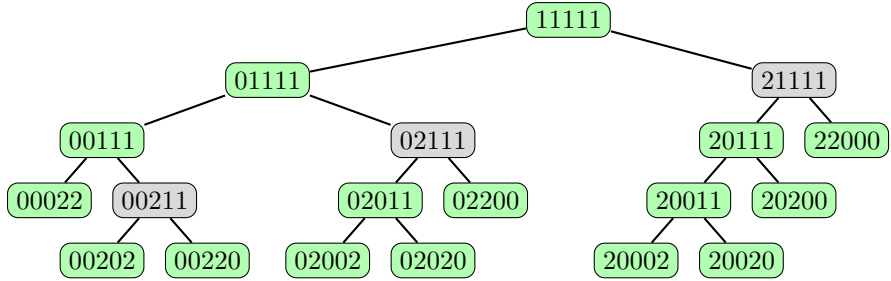
\begin{figure}
    \centering
    \begin{tikzpicture}[node distance=0.8cm]
\node (start) [node2] {11111};
\node (l11) [node2, below of=start, xshift=-4cm]{01111};
\node (l12) [node1, below of=start, xshift=3cm]{21111};
\node (l21) [node2, below of=l11, xshift=-2.2cm]{00111};
\node (l22) [node1, below of=l11, xshift=2.2cm]{02111};
\node (l23) [node2, below of=l12, xshift=-0.7cm] {20111};
\node (l24) [node2, below of=l12, xshift=0.7cm]{22000};
\node (l31) [node2, below of=l21, xshift=-0.7cm]{00022};
\node (l32) [node1, below of=l21, xshift=0.7cm]{00211};
\node (l33) [node2, below of=l22, xshift=-0.7cm]{02011};
\node (l34) [node2, below of=l22, xshift=0.7cm]{02200};
\node (l35) [node2, below of=l23, xshift=-0.7cm]{20011};
\node (l36) [node2, below of=l23, xshift=0.7cm]{20200};
\node (l41) [node2, below of=l32, xshift=-0.7cm]{00202};
\node (l42) [node2, below of=l32, xshift=0.7cm]{00220};
\node (l43) [node2, below of=l33, xshift=-0.7cm]{02002};
\node (l44) [node2, below of=l33, xshift=0.7cm]{02020};
\node (l45) [node2, below of=l35, xshift=-0.7cm]{20002};
\node (l46) [node2, below of=l35, xshift=0.7cm]{20020};
\draw [arrow] (start) -- (l11);
\draw [arrow] (start) -- (l12);
\draw [arrow] (l11) -- (l21);
\draw [arrow] (l11) -- (l22);
\draw [arrow] (l12) -- (l23);
\draw [arrow] (l12) -- (l24);
\draw [arrow] (l21) -- (l31);
\draw [arrow] (l21) -- (l32);
\draw [arrow] (l22) -- (l33);
\draw [arrow] (l22) -- (l34);
\draw [arrow] (l23) -- (l35);
\draw [arrow] (l23) -- (l36);
\draw [arrow] (l32) -- (l41);
\draw [arrow] (l32) -- (l42);
\draw [arrow] (l33) -- (l43);
\draw [arrow] (l33) -- (l44);
\draw [arrow] (l35) -- (l45);
\draw [arrow] (l35) -- (l46);
\end{tikzpicture}    
    \caption{IBB$^{+}$ tree for $p=5$ and $k=2$.}
    \label{fig:IvlBBtree}
\end{figure}
\begin{remark}
 Assuming no bound based deletions, IBB$^{+}$ and BB using ``minimum-solution-tree'' approach without in-level node ordering procedure will take exactly $\binom{p+1}{k+1}-\binom{p-1}{k+1}$ number of $lb\, f(\cdot)$ calls to reach the optimal solution. 
\end{remark} 
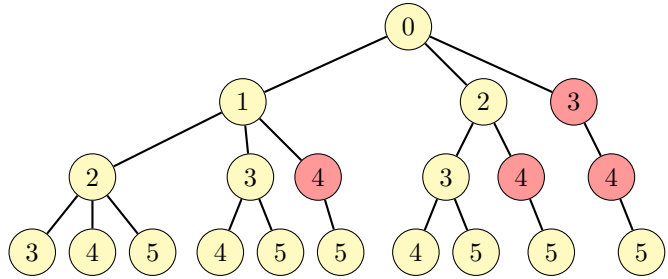
\begin{figure}
    \centering
    \begin{tikzpicture}[node distance=1cm]
\node (start) [node4]{0};
\node (d11) [node4, below of=start, xshift=-2.2cm]{1};
\node (d12) [node4, below of=start, xshift=1cm]{2};
\node (d13) [node3, below of=start, xshift=2.2cm]{3};
\node (d21) [node4, below of=d11, xshift=-2cm]{2};
\node (d22) [node4, below of=d11, xshift=0.1cm]{3};
\node (d23) [node3, below of=d11, xshift=1cm]{4};
\node (d24) [node4, below of=d12, xshift=-0.5cm]{3};
\node (d25) [node3, below of=d12, xshift=0.5cm]{4};
\node (d26) [node3, below of=d13, xshift=0.5cm]{4};
\node (d31) [node4, below of=d21, xshift=-0.8cm]{3};
\node (d32) [node4, below of=d21, xshift=0cm]{4};
\node (d33) [node4, below of=d21, xshift=0.8cm]{5};
\node (d34) [node4, below of=d22, xshift=-0.4cm]{4};
\node (d35) [node4, below of=d22, xshift=0.4cm]{5};
\node (d36) [node4, below of=d23, xshift=0.3cm]{5};
\node (d37) [node4, below of=d24, xshift=-0.4cm]{4};
\node (d38) [node4, below of=d24, xshift=0.4cm]{5};
\node (d39) [node4, below of=d25, xshift=0.4cm]{5};
\node (d40) [node4, below of=d26, xshift=0.4cm]{5};
\draw [arrow] (start) -- (d11);
\draw [arrow] (start) -- (d12);
\draw[arrow] (start) -- (d13);
\draw [arrow] (d11) -- (d21);
\draw [arrow] (d11) -- (d22);
\draw [arrow] (d11) -- (d23);
\draw [arrow] (d12) -- (d24);
\draw [arrow] (d12) -- (d25);
\draw [arrow] (d13) -- (d26);
\draw [arrow] (d21) -- (d31);
\draw [arrow] (d21) -- (d32);
\draw [arrow] (d21) -- (d33);
\draw [arrow] (d22) -- (d34);
\draw [arrow] (d22) -- (d35);
\draw [arrow] (d23) -- (d36);
\draw [arrow] (d24) -- (d37);
\draw [arrow] (d24) -- (d38);
\draw [arrow] (d25) -- (d39);
\draw [arrow] (d26) -- (d40);
\end{tikzpicture}
    \caption{BB tree for $p=5$ and $k=2$.}
    \label{fig:BBtree}
\end{figure}
\begin{table}
\centering
\caption{Number of OLS solutions needed to get to the optimal solution for IBB$^{+}$ and BB with in-level node ordering for OD examples of small-1 and small-2 type with $k\in \{5,10 \}$.\label{tab:ols_calls_ibb_vs_bb}}
\begin{tabular}{l l l l l l}
\hline
& &\multicolumn{2}{c}{k=10} & \multicolumn{2}{c}{k=5} \\
\cline{3-4}\cline{5-6}%
Example type & SNR & IBB$^{+}$  &   BB  &  IBB$^{+}$  &  BB \\
\hline
small-1 & 0.05  &    658    &    1036  &  378   &     1134 \\
   small-1 & 0.5   &    752    &    1148  &  381   &     1106 \\
    small-1 & 1    &    742    &    1139  &  367   &     1080 \\
    small-1 & 5    &    975    &    1554  &  254   &      773  \\
    small-2 & 0.05  &    39653   &  210179  & 7763   &  56131  \\ 
    small-2 & 0.5   &    39001   &  206137  &  7424  &  53632  \\
    small-2 & 1     &    38900   &  205324  &  7299  &  52504  \\
    small-2 & 5     &    39778   &  210140  &  6987  &  50729  \\
\hline
\end{tabular}
\end{table}
\section{Numerical results}
\label{sec:test_results}
In this section, we provide numerical results comparing the following three algorithms to solve the \eqref{bbss} problem. 
\begin{labeling}{AAAA}  
    \item[\textbf{IBB$^{+}$}] The non-interval version of Algorithm \ref{alg:ibb-bbss}, using an integer flag array to represent the box;
    \item[\textbf{BB}] The branch-and-bound algorithm following the implementation given in \cite{somol2004fast} with ``minimum-solution-tree'' strategy included as suggested in \cite{yu1993more};
    \item[\textbf{MIO}] The mixed integer optimization formulation solved by GUROBI.
\end{labeling} 
We note that IBB$^{+}$, BB, and MIO are exact methods to solve \eqref{bbss} by finding a globally optimal solution if no hard-stopping conditions are imposed. For MIO approach, we followed the implementation given in \cite{bertsimasEtal:2015} for setting up the MIO formulation. For $p\leq n$ and $p>n$ cases, we use two different MIO formulations given by equations 2.5 and 2.6 in \cite{bertsimasEtal:2015}, respectively.  After converting \eqref{bbss} to an equivalent MIO formulation, we use the \textsc{GUROBI} solver with its presolve feature turned off and restricting the number of threads to one for better comparison, leaving all the other parameters set to their default values. 
All the algorithms have been implemented on \textsc{MATLAB} R2021a, and testing has been done on The University of Alabama High-Performance Computer. For the MIO approach, \textsc{GUROBI}(\cite{gurobi}) has been called using the \textsc{MATLAB} interface for \textsc{GUROBI} 9.0.2. Selection of a new node at each iteration in IBB$^{+}$ algorithm is done using the BFS criteria. 
Also, the coordinate direction $\eta$ to partition the box $Y$ to obtain child boxes (step \ref{alg:ibb_partDir}) is done by choosing $\eta=\{i: \Hat{x}_{i}=\max\,|\Hat{x}|\}$, where $f(\Hat{x})=lb\,f(Y)$. The $lb\, f(\cdot)$ in step \ref{alg:ibb_boundstep} is computed using the parallel tangent method to solve convex quadratic optimization problems.
\subsection{Choosing the initial box}
Reference \cite{bertsimasEtal:2015} provided both theoretical and data-based approaches to choose the bound $M$ used in their MIO formulation. In our testing, instead of using a uniform bound for $\beta$ we use a non-uniform bound to define the box in \eqref{bbss}. We first find a solution $\Hat{\beta}$ to \eqref{ols} and define $m=max (|\Hat{\beta}_{i}|,i=1,...,p)$. Then the box $B=B_{1} \times  ... \times  B_{p}$  is defined as 
\[ - \tau m -|\Hat{\beta}_{i}| \leq B_{i} \leq |\Hat{\beta}_{i}|+ \tau m , \]
where $i\in \{1,...,p\}$ and $\tau$ is an enlargement factor that has been set to $\tau=1$.
For a fair comparison, we use the same box for all three algorithms (i.e. whenever we have to find a solution to \eqref{ols} problem at any step of these algorithms, we use this box). 
Note that the initial box $B$ is not known a \textit{priori}. If the chosen box $B$ is not big enough to contain the solution of \eqref{bss}, then the optimal solution of \eqref{bbss} may not be the same as the optimal solution of \eqref{bss}. However, if we do not use a box to find a solution to the \eqref{ols} problem within IBB$^{+}$ and BB, we can solve \eqref{bss} directly. While a bounded initial box is necessary for MIO solver of \eqref{bss}, IBB$^{+}$ and BB remain applicable for solving \eqref{bss} directly as long as \eqref{ols} solver does not require a bounded search space. In particular, the parallel tangent method used in our IBB$^{+}$ works better over the entire Euclidean space.
\subsection{Test data setup}
We have tested the three methods using synthetic data sets constructed as follows. Firstly, we find the design matrix $X \in \mathbb{R}^{n \times p}$ by sampling each row (i.i.d.) from a $p$-dimensional multivariate normal distribution $N(0,\Sigma)$, with mean zero and covariance matrix $\Sigma$. We normalize the columns of $X$ to have zero mean and unit $l_{2}$-norm. We construct the coefficient vector $\beta^{0}$. We choose a noise vector $\varepsilon$ (i.i.d.) from the normal distribution $N(0,\sigma^{2})$, where the variance $\sigma^{2}$ is chosen according to the given signal-to-noise (SNR) ratio defined as
\( \text{SNR}:=\frac{\|X\beta^{0} \|_{2}^{2}}{ \sigma^{2} }.  \)
Finally, we get the response vector $y$ using $y=X\beta^{0}+\varepsilon$. 
\tablename{ \ref{tab:sma_med_lar_eg_setup}} shows test examples we used in 3 different dimension groups for the overdetermined (OD, $p<n$) case and the underdetermined (UD, $p>n$) case. We have tested three groups of examples. We chose $k_{0}$ (the number of non-zero entries in $\beta^{0}$) to be 10, and the covariance matrix $\Sigma$ is such that $\Sigma_{i,j}=0.8$ when $i\neq j$ and $\Sigma_{i,i}=1$. We want to select the best model with 5 and 10 predictors, i.e. $k \in \{5,10 \}$. We tested three examples by varying $\beta^{0}$.
\begin{table}
\centering
\caption{Test examples dimension setup.\label{tab:sma_med_lar_eg_setup}}
\begin{tabular}{l l l l }
     \hline
     Type & $p$ & $n$ for OD & $n$ for UD   \\
     \hline
     small-1 & 20 & 100 & 10 \\
     small-2 & 40 & 200 & 20 \\
     small-3 & 60 & 300 & 30 \\
     small-4 & 80 & 400 & 40 \\
     \hline
     medium-1 & 200 & 1000 & 100 \\
     medium-2 & 300 & 1000 & 100 \\
     medium-3 & 400 & 2000 & 100 \\
     medium-4 & 500 & 2000 & 100 \\
     \hline
     large-1 & 800 & 4000 & 200 \\
     large-2 & 1000 & 4000 & 200 \\
     large-3 & 1500 & 8000 & 300 \\
     large-4 & 2000 & 8000 & 300 \\
    \hline
\end{tabular}
\end{table}
\begin{example}
 Generate $\beta^{0}$ such that $\beta_{i}^{0}=1$ for $k_{0}$ equally spaced indices from the set $\{1,...,p \}$, rounding to the greatest integer if needed.    
\end{example}

\begin{example}
Generate $\beta^{0}$ by assigning the first $k_{0}$ entries as $1$, that is $\beta_{i}^{0}=1$ for $i=1,...,k_{0}$. 
\end{example}

\begin{example}
Generate $\beta^{0}$ by picking a random subset from $\{1,...,p\}$ of $k_{0}$ indices, and then we assign random integer values between $1$ and $5$ to those indexed variables.  
\end{example}
We run the three examples in small, medium, and large dimensional settings as given in \tablename{ \ref{tab:sma_med_lar_eg_setup}} with 4 SNR values SNR $\in$ $ \{0.05, 0.5, 1, 5 \}$. 
We compare two aspects of these algorithms: solution quality and CPU time.
\subsection{Performance profiles and box plots}
To visualize the output data conveniently, we adopt two commonly used types of plots: performance profiles and box plots. For any number of examples tested and any predetermined performance measure, each solver could be associated with any such plot. By grouping these plots together, we can then visually compare all solvers.

Performance profiles as introduced in \cite{dolan2002perprof} use the idea of comparing the ratio of one solver$'$s performance measure with the best (minimum) performance measure among all solvers for that problem. More specifically, for each problem $p$ and solver $s$, let $t_{p,s}$ define a performance measure that we want to compare. We calculate the performance ratio as 
\[ r_{p,s}=\frac{t_{p,s}}{\min \{t_{p,s}:s\in S \} }\]
where $S$ is the set of solvers. We plot the resulting data using an empirical CDF plot. 

Box plots for sample data give us a visualization of five summary statistics of any given performance measure. The bottom and top of each box are the 25th and 75th percentiles of the sample. The horizontal line in the middle of the box shows the median of the sample data. The vertical lines extending from the box go to the upper and lower extreme. The observations beyond the upper and lower extremes are marked as outliers. For all the box plots in this paper, a value is called an outlier (marked as $+$ sign) if it is more than 1.5 times the interquartile range away from the bottom or top of the box.  
\subsection{Relative gap percentage as a performance measure}
To compare the solution quality, we use the Relative Gap $\%$ defined as 
\[ \left(  \frac{\Tilde{f}- f^{*}}{f^{*}} \right)100,\] 
where, for a particular example, $f^{*}$ is the best function value found by any algorithm and $\Tilde{f}$ is the function value from a given algorithm. A small Relative Gap $\%$ implies that the solution from a given algorithm is close to the best solution. 

Figures \ref{fig:boxPlotRelGapOdIbbBbMio} and \ref{fig:boxPlotRelGapUdIbbBbMio} show box plots of the Relative Gap $\%$ for the three examples in small, medium, and large dimension regimes in OD and UD cases, respectively. The general trend is that MIO always provides solutions with the lowest Relative Gap $\%$, followed closely by IBB$^{+}$. BB is the worst among the three algorithms. For small-dimension OD examples, both IBB$^{+}$ and BB are close to the best solution. However, for medium and large-dimension OD examples, IBB$^{+}$ is clearly better than BB, with BB not being able to provide the best solution for any of the large-dimension OD examples. UD examples are more computationally challenging than OD examples because of the greater randomness in the data and the lower chance of bound-based deletions for every branch and bound algorithm. BB is a lot worse than IBB$^{+}$ in this case. Note that in the UD case, due to the special formulation, the dimension of the problem to solve for MIO is $min(p,n)$, which gives MIO an advantage over IBB$^{+}$ and BB algorithms.        
\begin{figure}
    \centering
    \includegraphics[width=0.7\textwidth]{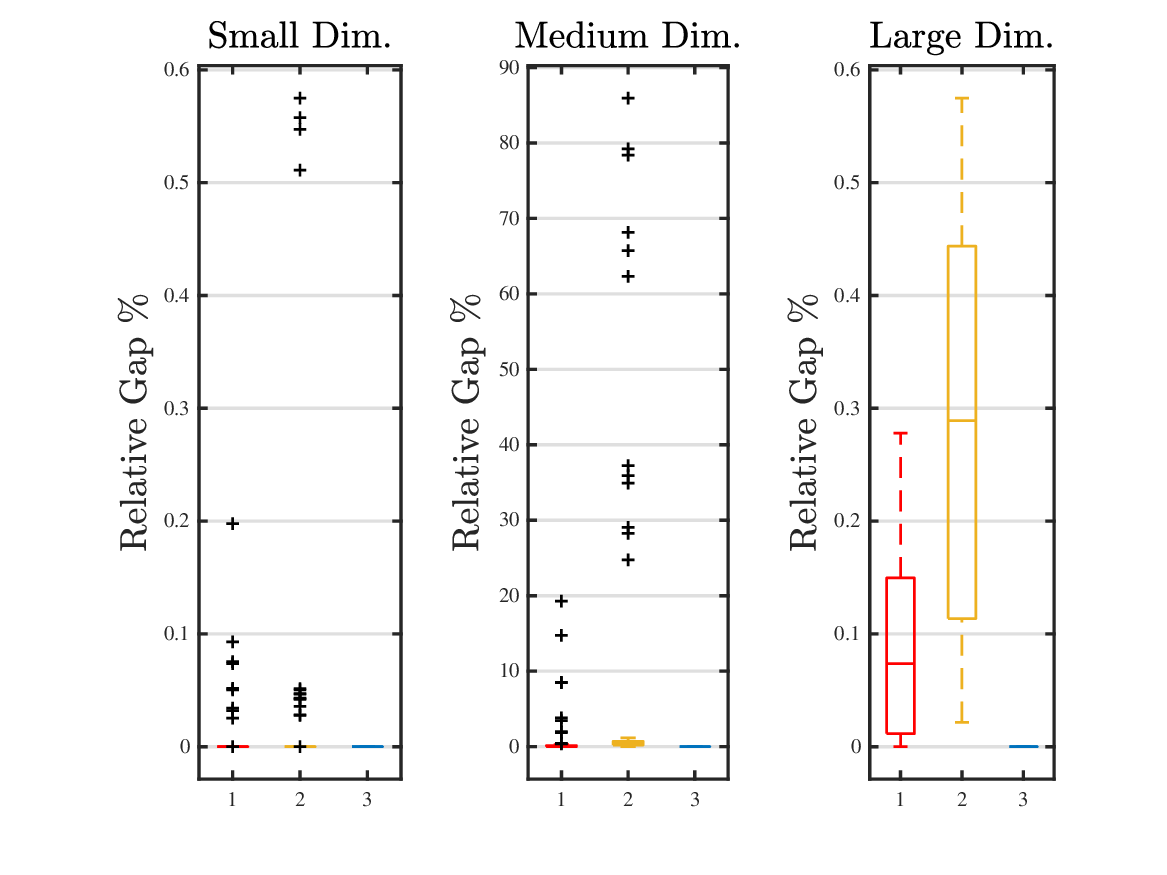}
    \caption{Box plots of Relative Gap $\%$ for examples 1, 2, and 3 in small, medium, and large dimension regimes with 4 SNR values in OD case with $k\in \{5,10 \}$ for 1) IBB$^{+}$; 2) BB; 3) MIO. }
    \label{fig:boxPlotRelGapOdIbbBbMio}
\end{figure}
\begin{figure}
    \centering
    \includegraphics[width=0.7\textwidth]{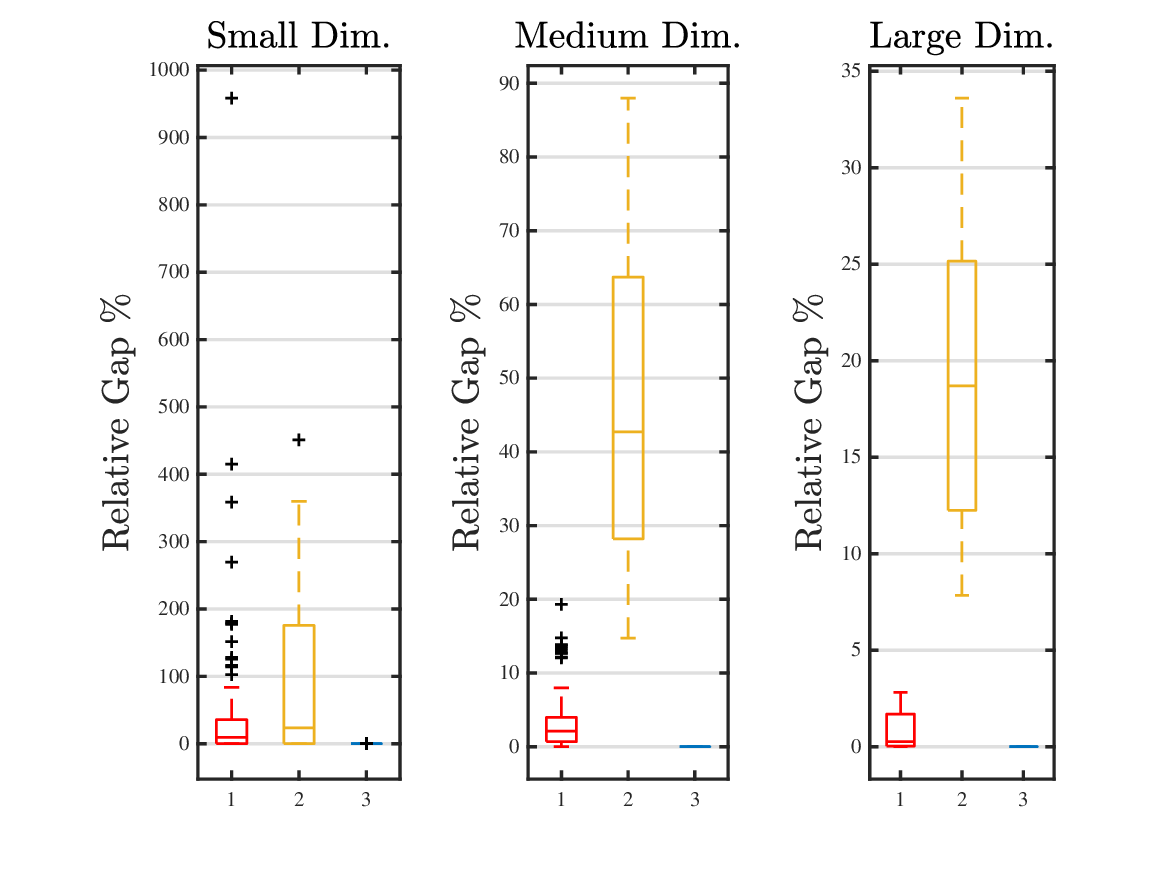}
    \caption{Box plots of Relative Gap $\%$ for examples 1, 2, and 3 in small, medium, and large dimension regimes with 4 SNR values in UD case with $k\in \{5,10 \}$ for 1) IBB$^{+}$; 2) BB; 3) MIO. }
    \label{fig:boxPlotRelGapUdIbbBbMio}
\end{figure}
\subsection{CPU time as a performance measure}
It is well known that finding a globally optimal solution using an algorithm can take a lot of CPU time. Even if the algorithm finds the optimal solution quickly, the only way to certify its optimality is to go over the space of all the possible solutions. In our testing, we are using a hard CPU time limit of $10$ minutes, meaning that after $10$ minutes of CPU time, we will stop the algorithm and report the current best solution as the final solution for that algorithm. Also, for IBB$^{+}$ and BB algorithms, we set the maximum iteration limit to be 10,00,000 and we stop the algorithm once this limit has been reached (reporting the current best solution as the final solution). We have incorporated two soft stops in IBB$^{+}$ along with these hard limits. We will stop the algorithm if $\tilde{f}$ in IBB$^{+}$ does not improve for $500$ iterations or for $5$ minutes of CPU time. With these soft stops included, IBB$^{+}$ can provide a solution close to optimal with much less CPU time.

Figures \ref{fig:perprofOdEgsCpuIbbBbMio} and \ref{fig:perprofUdEgsCpuIbbBbMio} show the performance profiles of CPU time for the three examples in small, medium, and large dimension regimes in OD and UD cases, respectively.
Due to the soft stopping criteria included in IBB$^{+}$, it is the fastest among the three algorithms. For small dimension examples, IBB$^{+}$ and MIO are close. BB is the worst of the three algorithms in terms of CPU time and generally stops due to the hard CPU time limit of 10 minutes.  
\begin{figure}
    \centering
    \includegraphics[width=0.7\linewidth]{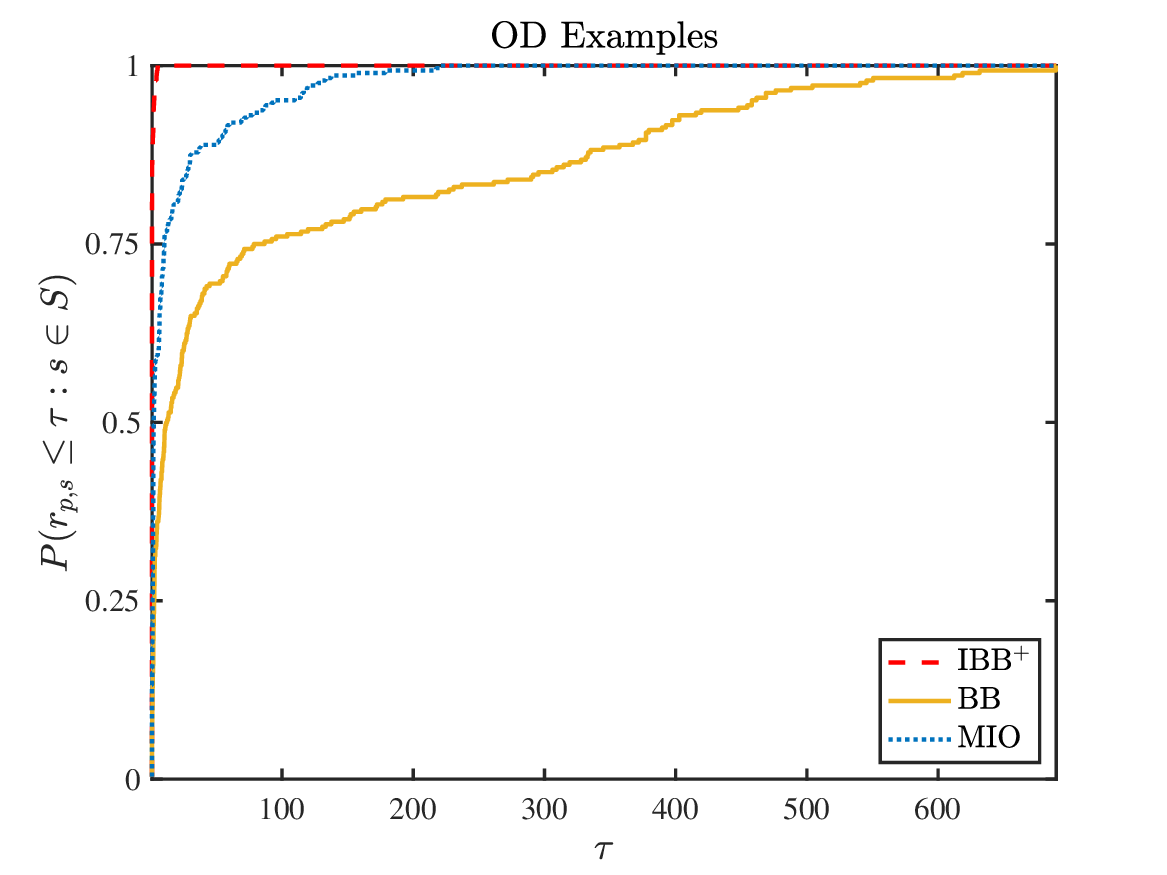}
    \caption{Performance profiles of CPU time for examples 1, 2, and 3 in small, medium, and large dimension regimes with 4 SNR values in OD case and $k\in \{5,10 \}$.}
    \label{fig:perprofOdEgsCpuIbbBbMio}
\end{figure}
\begin{figure}
    \centering
    \includegraphics[width=0.7\linewidth]{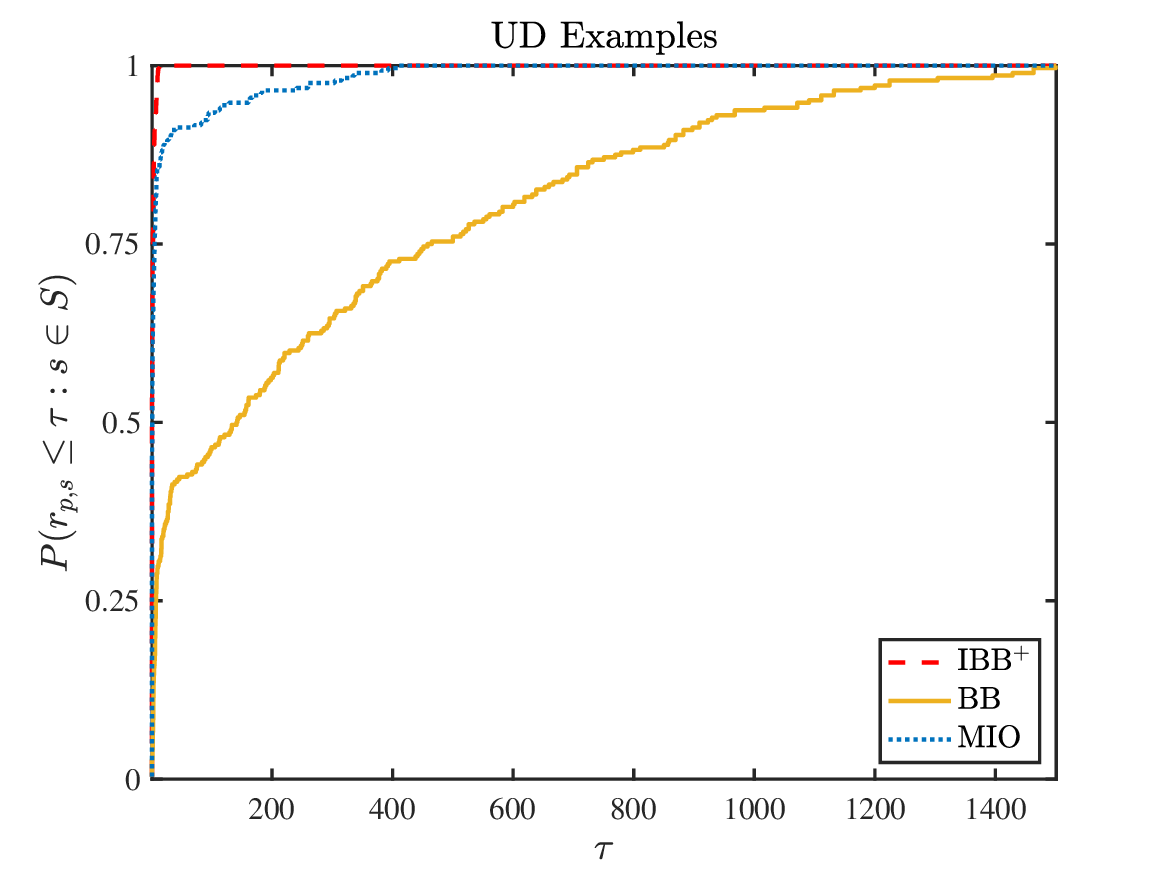}
    \caption{Performance profiles of CPU time for examples 1, 2, and 3 in small, medium, and large dimension regimes with 4 SNR values in UD case and $k\in \{5,10 \}$.}
    \label{fig:perprofUdEgsCpuIbbBbMio}
\end{figure}
\section{Conclusion}
\label{sec:discussion}
In this paper, we introduced IBB$^{+}$ that follows the framework of interval branch-and-bound method to solve \eqref{ccqo} problem. IBB$^{+}$ finds a globally optimal solution of \eqref{ccqo} using a special enumeration to go over all the possible subsets and discard those that cannot contain an optimal solution using some deletion conditions. We applied IBB$^{+}$ to solve the best subset selection problem in regression and compared it with two other exact methods, BB and MIO. The numerical results show that IBB$^{+}$ outperforms BB and is competitive with MIO. We can further include some acceleration strategies within IBB$^{+}$ to make it faster while keeping the global convergence of the algorithm. IBB$^{+}$ does not need the convexity of the objective function as an assumption as long as we are using tight lower bounds within the algorithm. The proposed algorithm can also accommodate linear inequality constraints in \eqref{ccqo} by making appropriate adjustments within IBB$^{+}$ but still maintaining the global convergence of the algorithm.  
\section*{Acknowledgement}
\small
We would like to thank Ward Jaeger for his contribution in implementing the heap data structure. We also thank The University of Alabama and the Office of Information Technology for providing high-performance computing resources and support that have contributed to these research results.
\section*{Reproducibility}
\small
 The test results in this paper can be reproduced by downloading the corresponding files from  
 \href{https://github.com/vikrasingh/bss-ibb-bb-mio}{https://github.com/vikrasingh/bss-ibb-bb-mio}.
\bibliographystyle{plain} 
\bibliography{main} 

\end{document}